\documentclass{article}

\usepackage{amsmath, amssymb, latexsym, amsthm, xspace, color}
\usepackage{ifpdf}
\usepackage{paralist}
\usepackage{graphicx}
\usepackage{subfigure}
\usepackage[all]{xy}

\ifpdf
\else
\usepackage{epsfig}
\fi

\numberwithin{equation}{section}

\newcommand{\cx}{{\mathbb C}} 
\newcommand{\half}{{\mathbb H}} 
 
\newcommand{\integers}{{\mathbb Z}}

\newcommand{\ratls}{{\mathbb Q}} 
 
\newcommand{\proj}{{\mathbb P}}
\newcommand{\zed}{\integers}

\newcommand{\barmoduli}[1][g]{{\overline{\mathcal M}}_{#1}}
\newcommand{\bdry}{\partial}

\newcommand{\Gm}{\mathbb{G}_m}

\renewcommand{\tilde}{\widetilde}

\newcommand{\RM}[1][\mathcal{O}]{{\cal RM}_{#1}}

\newcommand{\RA}[1][\mathcal{O}]{{\cal RA}_{#1}}
\newcommand{\RS}[1][E]{{\cal RS}_{#1}}

\newcommand{\moduli}[1][g]{{\mathcal M}_{#1}}
\newcommand{\AVmoduli}[1][g]{{\mathcal A}_{#1}}

\newcommand{\A}[1][g]{\AVmoduli[#1]}

\newcommand{\SL}{{\mathrm{SL}}}

\newcommand{\Ord}{\mathcal{O}}

\newcommand{\ba}{{\boldsymbol{a}}}

\newcommand{\bzero}{\boldsymbol{0}}

\DeclareMathOperator{\Hom}{Hom}

\DeclareMathOperator{\Ann}{Ann}

\DeclareMathOperator{\CR}{CR}
\DeclareMathOperator{\Hol}{Hol}

\DeclareMathOperator{\Sym}{Sym}
\DeclareMathOperator{\bS}{{\bf S}}
\DeclareMathOperator{\Tr}{Tr}

\DeclareMathOperator{\trace}{Tr}

\newtheoremstyle{example}{3pt}{3pt}{\upshape}{}{\itshape}{.}{.5em}{}

\newtheorem{theorem}{Theorem}[section] 
\newtheorem{prop}[theorem]{Proposition} 
\newtheorem{cor}[theorem]{Corollary}
\newtheorem{lemma}[theorem]{Lemma}

\newcommand{\tR}{{\tilde{R}}}

\theoremstyle{example}

\theoremstyle{definition}

\theoremstyle{remark}

\makeatletter
\def\blfootnote{\xdef\@thefnmark{}\@footnotetext}

\renewcommand{\l@section}{\@dottedtocline{0}{1.5em}{2.3em}}
\renewcommand{\l@subsection}{\@dottedtocline{1}{3.8em}{3.2em}}
\renewcommand{\l@subsubsection}{\@dottedtocline{2}{7.0em}{4.1em}}

\makeatother

\newcommand{\cS}{\mathcal{S}}

\newcommand{\ev}{{\rm ev}}
\newcommand{\ii}{{\rm in}}
\newcommand{\out}{{\rm out}}

\begin{document}

\bibliographystyle{halpha}

\title{The locus of real multiplication and \\ the Schottky locus}
\author{Matt Bainbridge and Martin M\"oller}
\maketitle

\tableofcontents

\section{Introduction}
\label{sec:introduction}

Are there locally symmetric subvarieties in the moduli space of
abelian varieties, whose generic point lies in the image of
the moduli space of curves, i.e.\ in the Schottky locus? 
This question was raised by Oort, motivated by the Conjecture
of Coleman that for $g$ large enough there are only finitely 
many curves of genus $g$ with complex multiplication.
\par
The first important contributions to this problem were made by Hain \cite{Hain},
whose results were subsequently refined by deJong and Zhang \cite{deJZhang}, 
see also   \cite{MoonenOort} for a survey.
In this paper we do not consider general
locally symmetric subvarieties (or Shimura subvarieties) but
restrict to the case of Hilbert modular varieties. They parametrize
abelian varieties with real multiplication. So our main result
is a statement about components of the real multiplication locus.
\par 
\begin{theorem}
  \label{thm:Intro}
  There is no component of the real multiplication locus in the moduli
  space of four-dimensional abelian varieties $\AVmoduli[4]$ that lies
  generically in the image of the moduli space of curves $\moduli[4]$.
\end{theorem}
\par
In \cite{deJZhang} the analogous theorem is proved for genus greater than four
with part of the genus four case still left open. Together, the 
two results imply:
\par
\begin{cor}
  \label{cor:total}
  For every genus $g$ and every component of the real multiplication locus in the moduli
  space of $g$-dimsensional abelian varieties $\AVmoduli[g]$, the generic
point of the component does not lie in the image of the moduli space of curves $\moduli[g]$.
\end{cor}
\par
Besides completing the work of \cite{deJZhang}, we believe
our method of proof is interesting for the following reason. 
\par
The proofs in \cite{deJZhang} and \cite{Hain} ultimately rely on a
theorem of Farb and Masur \cite{farbmasur} that mapping class groups
do not contain fundamental groups of lattices in higher rank Lie
groups. Consequently, if the generic point of a component of the real
multiplication locus lies in the moduli space of curves, then the
Torelli map must modify the fundamental group of the real
multiplication locus, either by ramification along the hyperelliptic
locus, or by the locus of decomposable abelian varieties, which is
disjoint from the Schottky locus.  Whenever this can be ruled out,
e.g.\ by showing that the codimension of this intersection is at least
two, one has the desired contradiction.
\par
The proof here, on the contrary, relies on a study at the
boundary of the moduli space of curves. Since Hilbert modular varieties
are not compact, we may study their closure in the Deligne-Mumford
compactification. A counterexample to the theorem in genus four 
must have a component of dimension three in some Deligne-Mumford
boundary stratum of the moduli space of curves.
\par
These closures of Hilbert modular varieties were analyzed
in \cite{bainmoel}. If one uses cross-ratios as degenerate
period coordinates, the closure of a Hilbert modular variety
is contained in a subtorus of an ambient algebraic torus 
which we call the {\em RM-torus}.
So a first try to rule out counterexamples to the main theorem 
is to check if the images of Deligne-Mumford boundary strata in
cross-ratio coordinates contain tori of sufficiently large dimension.
\par
In fact, they do contain such large tori, but only when 
the tori are very degenerate, e.g.\ lying completely in
a coordinate hyperplane. The heart of the paper consists
in showing that RM-tori do not have this property. For that purpose, following
the ideas in  \cite{bainmoel}, we provide
the (dual graph of the) boundary stratum with {\em weights}
given by the residues of an eigenform for real multiplication. Only if
the weights satisfy the restrictive condition of
being admissible, the boundary stratum can lie in the closure
of the Hilbert modular variety.  This condition is
recalled in Theorem~\ref{thm:RMboundary}. 
\par
The obvious refinement of the above theorem, 
to understand the {\em dimension} of the real multiplication 
locus in $\AVmoduli$ with $\moduli$ (say for large $g$) is still an 
open problem. The techniques in this paper could contribute
to the solution of this problem, which is not tractable by 
methods based on the lattice properties of fundamental group.
\par
Section~\ref{sec:RM-tori} derives the key properties of RM-tori, 
e.g.\ a method to calculate their intersection dimension with
subtori in terms of the weights.
Section~\ref{sec:redtobd} contains the details of the strategy outlined above.
The main theorem is reduced in that section to showing 
for a list of relevant Deligne-Mumford boundary strata that the
RM-tori are not contained in the image of the stratum under
the cross-ratio maps. 
In Section~\ref{sec:CRnice} we show that cross-ratios are indeed coordinates
near the boundary of a relevant boundary stratum if this
boundary stratum does not parametrize curves with a separating node.
In Section~\ref{sec:check} we give graph-theoretic criteria for
the desired non-containment statement 
and thus complete the argument for relevant boundary strata
parameterizing curves without a separating node.
In the last section we deal with boundary strata that parametrize curves 
with a separating node and reduce this to a case previously dealt with.
\par

\section{Boundary of the real multiplication locus}
\label{sec:boundary}

In this section, we summarize properties of the real multiplication locus and their boundaries which
will be needed in later sections.  Throughout this section, $F$ will denote a totally real number
field of degree $g$, and $\mathcal{O}$ will denote an order in $F$ (a subring which has rank $g$ as
an Abelian group).

\paragraph{The real multiplication locus.}

We denote by $\RA\subset\A$ the locus of
Abelian varieties which have real multiplication by $\mathcal{O}$.  This locus is an immersed quotient of a
Hilbert modular variety by a finite group of automorphisms.  We denote by $\RM\subset\moduli$ the
locus of Riemann surfaces whose Jacobians have real multiplication by $\mathcal{O}$.  In other
words, $\RM = t^{-1}(\RA)$, where $t \colon \moduli\to\A$ is the Torelli map.

\paragraph{Irreducible components.}

Given an Abelian variety $A$ with real multiplication $\mathcal{O}$, the homology group $H_1(A;
\zed)$ has the structure of an $\mathcal{O}$-module with a compatible symplectic structure, and the
isomorphism classes of such modules parametrize the irreducible components of $\RA$.

More precisely, consider a torsion-free $\mathcal{O}$-module $M$.  The \emph{rank} of $M$ is the
dimension of $M\otimes \ratls$ as a vector space over $F$.  We say that $M$ is proper if the
$\mathcal{O}$-module structure doesn't extend to a larger order.  A \emph{symplectic
  $\mathcal{O}$-module} is a torsion-free $\mathcal{O}$-module equipped with a unimodular symplectic
form satisfying $\langle x, \lambda y\rangle = \langle \lambda x, y\rangle$ for each $x,y\in M$ and
$\lambda\in \mathcal{O}$.

Given a torsion-free, proper, rank-two, symplectic $\mathcal{O}$-module $M$, we define
$\RA[\mathcal{O}, M]\subset\RA$ to be the locus of Abelian varieties whose first homology is
isomorphic to $M$ as a symplectic $\mathcal{O}$-module.  $\RA[\mathcal{O}, M]$ is isomorphic to a
finite quotient of a Hilbert modular variety $\half^g/\Gamma$ for some $\Gamma$ commensurable with
$\SL_2(\mathcal{O})$, so it is irreducible.  Thus the irreducible components of $\RA$ are parametrized
by isomorphism classes of such $M$.

\paragraph{Cusps.}

A \emph{lattice} $\mathcal{I}$ in $F$ is a rank $g$ additive subgroup
$\mathcal{I}\subset F$.  The \emph{coefficient ring} of $\mathcal{I}$ is the order $\mathcal{O}_\mathcal{I}$ defined
by
\begin{equation*}
  \mathcal{O}_\mathcal{I} = \{a \in F : ax\in \mathcal{I} \text{ for all } x \in \mathcal{I}\}.
\end{equation*}

The \emph{inverse different} of $\mathcal{I}$ is the lattice $\mathcal{I}^\vee$ defined by
\begin{equation*}
  \mathcal{I}^\vee = \{x\in F : \langle x, y\rangle \in \zed \text{ for all } y \in \mathcal{I}\},
\end{equation*}
also having coefficient ring $\mathcal{O}_\mathcal{I}$.  Here and throughout the paper we use the
notation $\langle x, y\rangle = \Tr(xy)$ for the trace pairing.

Consider a rank-two symplectic $\mathcal{O}$-module $M$ and a lattice $\mathcal{I}$ whose
coefficient ring contains $\mathcal{O}$.  An exact sequence of $\mathcal{O}$-modules,
\begin{equation}
  \label{eq:2}
  0\to\mathcal{I}\to M \to \mathcal{I}^\vee \to 0,
\end{equation}
expresses $M$ as an extension of $\mathcal{I}^\vee$ by $\mathcal{I}$.  The sequence \eqref{eq:2}
splits as a sequence of Abelian groups, yielding a group isomorphism $\mathcal{I} \oplus
\mathcal{I}^\vee\to M$.  The module $\mathcal{I}\oplus\mathcal{I}^\vee$ carries a natural symplectic
structure, defined by
\begin{equation*}
  \langle (a, b), (c, d)\rangle = \Tr(ad - bc).
\end{equation*}
This induces a symplectic structure on $M$ which does not depend on the choice of splitting of
\eqref{eq:2}.  We define $E(\mathcal{I}, M)$ to be the set of isomorphism classes of extensions
\eqref{eq:2} such that the induced symplectic form on $M$ agrees with the given one.

\paragraph{Stable forms.}

Consider a stable curve $X$, and let $X'\subset X$ be the complement of the nodes.  A \emph{stable
  form} on $X$ is a holomorphic one-form on $X'$ which has at worst simple poles at the cusps of
$X'$, with opposite residues at two cusps which share a node.




\paragraph{Weighted stable curves.}

Consider a lattice $\mathcal{I}$ in a totally real number field $F$ of degree $g$.  An
$\mathcal{I}$-weighted stable curve is an arithmetic genus $g$, geometric genus $0$ stable curve
$X$, together with an element of $\mathcal{I}$ assigned to each cusp of $X'$ (the complement of the
nodes), called the weight of that cusp, subject to the following restrictions:
\begin{itemize}
\item Cusps of $X'$ sharing a node have opposite weight.
\item The sum of weights of a component of $X'$ is zero.
\item The weights span $\mathcal{I}$.
\end{itemize}
One could think of a weighted stable curve as a curve together with a stable form whose residues
belong to $F$.

When we do not care to specify the lattice $\mathcal{I}$, we may speak of a $F$-weighted stable
curve, or just a weighted stable curve.

Two weighted stable curves are \emph{isomorphic} (resp.\ \emph{topologically equivalent}) if there
is a weight preserving isomorphism (resp.\ homeomorphism) between the underlying stable curves.

If we don't want to specify the ideal (or the field) the weights span we
just talk of {\em $F$-weighted} (or just {\rm weighted}) stable curves, 
with the implicit meaning that the weights span a lattice $\mathcal{I}$.

\paragraph{Weighted boundary strata.}

We define an \emph{$\mathcal{I}$-weighted boundary stratum} $\mathcal{S}$ to be the moduli space of
$\mathcal{I}$-weighted stable curves which are topologically equivalent to some fixed weighted
stable curve $X$.  If $X$ has $m$ components, each having $n_i$ cusps with all weights distinct,
then the $\mathcal{S}$ is isomorphic to
\begin{equation*}
  \prod_i \moduli[0, n_i],
\end{equation*}
where $\moduli[0,n]$ is the moduli space of $n$ labelled points on $\proj^1$.  If some weights
coincide, the stratum may be a quotient of this product. There is a canonical morphism
$\mathcal{S}\to\barmoduli$ which forgets the weights.


An $\mathcal{I}$-weighted boundary stratum, or equivalently the topological type of a weighted
stable curve may be encoded by a directed graph with edges weighted by elements of $\mathcal{I}$.
Given a stratum $\mathcal{S}$ parameterizing weighted curves topologically equivalent to $X$, we
write $\Gamma(\mathcal{S})$ for the graph with one vertex for each component of $X$, with an edge
joining two vertices if the corresponding components are joined by a node (the \emph{dual graph}).
Contrary to usual practice, we allow graphs where an edge joins a vertex to itself, or where
multiple edges join the same pair of vertices.  We label each edge with the weight of the
corresponding node and an arrow pointing to the component with that weight (as opposed to its
negative).  We call such an object an \emph{$\mathcal{I}$-weighted graph}.  Two graphs which are
related by changing the orientation of an edge and simultaneously the sign of its weight represent
the same weighted boundary stratum, and we regard two such weighted graphs to be the same.
\par
A \emph{degeneration} of a weighted boundary stratum $\mathcal{S}$ is a stratum obtained by pinching
one or more simple closed curves on stable curves parametrized by $\mathcal{S}$.  A degeneration
$\mathcal{S}'$ of $\mathcal{S}$ can be regarded as part of the boundary of the Deligne-Mumford
compactifiction of $\mathcal{S}$.  On the level of dual graphs, degenerations of $\mathcal{S}$ are
obtained by gluing an edge into a vertex $v$ of $\Gamma(\mathcal{S})$.  More precisely, we replace
the vertex $v$ two vertices $v_1$ and $v_2$ joined by an edge $e$, with each edge meeting $v$ now
meeting either $v_1$ or $v_2$.  We assign $e$ the unique weight which is consistent with the axioms
of a weighted graph.

\paragraph{Periods.}

Consider an $\mathcal{I}$-weighted boundary stratum $\mathcal{S}$.  We recall here a coordinate-free
analogue of classical period matrices for weighted stable curves introduced in \cite{bainmoel}.

Given any ring $R$ and module $M$ over $R$, we define $\Sym_R(M)$ to be the submodule of $M
\otimes_R M$ fixed by the involution $\theta(x\otimes y) = y \otimes x$.  We define $\bS_R(M)$ to be
the quotient of $M\otimes_R M$ by the submodule generated by the relations $\theta(z) - z$.

We identify the field $F$ with its dual via the trace pairing; thus the vector
spaces $\Sym_\ratls(F)$ and $\bS_\ratls(F)$ are dual via the pairing
\begin{equation*}
  \langle a \otimes b, c \otimes d\rangle = \langle a, c\rangle\langle b, d\rangle,
\end{equation*}
as are the groups $\Sym_\zed(\mathcal{I})$ and
$\bS_\zed(\mathcal{I}^\vee)$. 

Let $W(\mathcal{S})\subset \Sym_\ratls(F)$ be the subspace generated by the elements $r \otimes r$
for $r$ running over the weights of $\mathcal{S}$.  Let $N(\mathcal{S})\subset\bS_\ratls(F)$ be the
annihilator of $W(\mathcal{S})$.

We defined in \cite{bainmoel} a homomorphism 
\begin{equation*}
  \Psi\colon N(\mathcal{S})\cap \bS_\zed(\mathcal{I}^\vee) \to \Hol^*(\mathcal{S}),
\end{equation*}
where the image is the multiplicative group of nonzero holomorphic
functions on $\mathcal{S}$.  Each $\Psi(x)$ is a holomorphic function
on $\mathcal{S}$ which arises as a limit of an exponential of a
classical period matrix entry.  We describe here $\Psi(x)$ when $x$ is
an elementary tensor $\alpha\otimes\beta$, and refer the reader to
\cite{bainmoel} for a more careful definition.

Consider $\alpha\otimes\beta\in N(\mathcal{S})\cap\bS_\zed(\mathcal{I}^\vee)$ and a weighted stable
curve $X\in\mathcal{S}$.  Pairing $\alpha$ with the weights of $X$ associates an integer to each
cusp of $X$.  There is a unique stable form on $X$ with these residues at the cusps, which
we call $\omega_X$.  Similarly, pairing $\beta$ with the weights associates an integer to each cusp,
and we may choose a path $\gamma$ on $X$ whose algebraic intersection number with each node is given
by these integers.  We define
\begin{equation} \label{eq:PsiDef}
  \Psi(\alpha\otimes\beta)(X) = e^{\int_\gamma \omega_X}.
\end{equation}
Since $\alpha\otimes\beta$ belongs to $N(\mathcal{S})$, we may choose $\gamma$ to not pass though
any nodes at which $\omega_X$ has a pole, so this integral is finite and well-defined.  It may be
checked directly that $\Psi(\alpha\otimes\beta) = \Psi(\beta\otimes\alpha)$, or more conceptually
this follows from the symmetry of period matrices of nonsingular curves by a degeneration argument.

The function $\Psi(x)$ is always a product of various cross-ratios of points on components the
stable curve. 

\paragraph{The necessary condition.}

We now recall the necessary condition for a stable curve to lie in the boundary of the real
 multiplication locus $\RM$.

Consider an $\mathcal{I}$-weighted boundary stratum $\mathcal{S}$.  In $\bS_\ratls(F)$, we define
the cone
\begin{equation*}
  C(\mathcal{S}) = \{ x \in \bS_\ratls(F) : \langle x, r \otimes r\rangle \geq 0 \text{ for each
    weight $r$ of $\mathcal{S}$}\}.
\end{equation*}

The space $F\otimes_\ratls F$ has the structure of an $F$-bimodule.  We define
\begin{equation*}
  \Lambda^1  = \{x \in F \otimes_\ratls F : \lambda\cdot x = x\cdot\lambda \text{
    for each $\lambda\in F$}\}.
\end{equation*}
In fact, $\Lambda^1$ is contained in $\Sym_\ratls(F)\subset F\otimes_\ratls F$ (see
\cite[Proposition~5.1]{bainmoel}).  We define $\Ann(\Lambda^1)\subset\bS_\ratls(F)$ to be the
annihilator of $\Lambda^1$.

We say that the weighted stratum $\mathcal{S}$ is \emph{admissible} if
\begin{equation*}
  C(\mathcal{S}) \cap \Ann(\Lambda^1) \subset N(\mathcal{S}).
\end{equation*}

We associate to an admissible stratum $\mathcal{S}$ various algebraic tori.  We define the
\emph{ambient torus} $A_\mathcal{S}$ by
\begin{equation*}
  A_\mathcal{S} = \Hom_\zed(N(\mathcal{S})\cap \mathcal{S}_\zed(\mathcal{I}^\vee), \Gm).
\end{equation*}
(Readers unfamiliar with algebraic groups should regard $\Gm$ as the
multiplicative group of nonzero complex numbers.)  The homomorphism
$\Psi\colon N(\mathcal{S})\cap
\bS_\zed(\mathcal{I}^\vee)\to\Hol^*(\mathcal{S})$ determines a
canonical morphism $\CR\colon\mathcal{S}\to A_\mathcal{S}.$

There is a surjective map of algebraic tori:
\begin{equation} \label{eq:defofp}
  p\colon A_\mathcal{S} \to \Hom(N(\mathcal{S})\cap
  \Ann(\Lambda^1)\cap \bS_\zed(\mathcal{I}^\vee), \Gm).
\end{equation}
We define the \emph{real multiplication torus} (or \emph{RM-torus}) $T_\mathcal{S}$ to be the subtorus
$T_\mathcal{S} = p^{-1}(0)\subset A_\mathcal{S}$.  More generally, there is a function
\begin{equation*}
  q\colon E(\mathcal{I},M) \to \Hom(N(\mathcal{S})\cap
  \Ann(\Lambda^1)\cap \bS_\zed(\mathcal{I}^\vee), \Gm),
\end{equation*}
with image in the set of torsion points, which is defined in \cite[\S5]{bainmoel}; we refer the
reader to that paper for the definition, as it is not needed here.  Given an extension $E\in
E(\mathcal{I},M)$, we define the \emph{translated RM-torus} $T_{\mathcal{S},E}$ by
$T_{\mathcal{S},E} = p^{-1}(q(E))\subset A_\mathcal{S}$.  Given an extension $E$, we define the
subvariety $\RS\subset\mathcal{S}$ to be the inverse image of $T_{\mathcal{S},E}$ under $\CR$.

We can now state our necessary condition for a geometric genus zero stable curve to lie in the
boundary of the real multiplication locus.  See \cite[\S5]{bainmoel} for the proof.

\begin{theorem}
  \label{thm:RMboundary}
  If a geometric genus zero stable curve $X\in\barmoduli$ lies in the boundary of 
  $\RM[\mathcal{O}, M]$, then there is a lattice $\mathcal{I}\subset F$ whose coefficient ring contains
  $\mathcal{O}$ and an extension $E\in E(\mathcal{I}, M)$ such that $X$ is in
  the image of $\RS$ under the forgetful map $\mathcal{S}\to \barmoduli$ for some
  admissible $\mathcal{I}$-weighted boundary stratum $\mathcal{S}$.
\end{theorem}

\paragraph{Loops in weighted graphs.}

Consider a geometric genus zero weighted stable curve $X$ lying in a stratum $\mathcal{S}$.  There
is a natural bijection between loops in the weighted graph $\Gamma(\mathcal{S})$ and homotopy
classes of loops on $X$.  We now describe a method to construct elements of $N(\cS)$ from pairs of
loops in the weighted graph $\Gamma(\mathcal{S})$.

We say that two loops in a graph are {\em edge-disjoint}, if they do not share an edge. If they do
not share a vertex, we call them {\em vertex-disjoint}.  We say that a loop is \emph{simple} if it
meets each vertex at most once.

To a loop $\gamma$ in $\Gamma(\mathcal{S})$, define
a functional $\lambda^*(\gamma)\in \Hom_\ratls(F, \ratls)$ so that for any edge $e$ having weight
$r$, $\lambda^*(\gamma)(r)$ is the number of times $\gamma$ traverses $e$ in the positive direction
minus the number of times $\gamma$ traverses $e$ in the negative direction.  This defines a 
functional on $F$ as the weights span $F$.  It is linear and well defined by the properties of a weighted
stable curve.

We define $\lambda(\gamma)\in F$ to be the unique element such that
$$\langle\lambda(\gamma),x\rangle= \lambda^*(\gamma)(x)$$
for all $x\in F$.  If we think of $\gamma$ as a loop on the stable curve $X$, and if $n$ is a node
of $X$ having weight $r$, then $\langle\lambda(\gamma), r\rangle$ is the algebraic intersection
number of $\gamma$ with $n$.  We may also regard $\lambda$ as an isomorphism $\lambda\colon
H_1(\Gamma(\mathcal{S}); \zed)\to \mathcal{I}^\vee$.

We can calculate $\lambda(\gamma)$ explicitly as follows.  Choose edges $e_1, \ldots, e_g$ of
$\Gamma(\mathcal{S})$ having corresponding weights $r_1, \ldots, r_g\in F$ such that these $r_i$
form a basis of $F$ over $\ratls$.  Let $s_1, \ldots, s_g$ be the dual basis of $F$ with respect to
the trace pairing.  Then $\lambda(\gamma) = \sum n_i s_i$, where $n_i$ is the number of times
$\gamma$ traverses $e_i$, crossings with the opposite orientation counted negatively.

If a simple loop $\gamma$ passes through a vertex $v$, we denote by
$\gamma^{\ii}(v)$ (resp.\ $\gamma^{\out}(v)$) the incoming (resp.\ outgoing)
marked point on the component of the stable curve corresponding to $v$.

Given $a,b,c,d\in\cx$, recall that their cross-ratio is defined by
\begin{equation*}
  [a,b,c,d] = \frac{(a-c)(b-d)}{(a-d)(b-c)}.
\end{equation*}
\par
\begin{lemma}
  \label{le:NSloops}
  Suppose the loops $\gamma_1$ and $\gamma_2$ in
  $\Gamma(\mathcal{S})$ are edge-disjoint.  Then the corresponding element $a=\lambda(\gamma_1)
  \otimes \lambda(\gamma_2) \in \bS(F)$ lies in $N(\cS)$. Moreover, if the $\gamma_i$ are simple,
  then the function $\Psi(a)$ is a product of cross-ratios
  \begin{equation*}
    \Psi(a) = \prod_{v \in \gamma_1 \cap \gamma_2}  
[\gamma_1^{\out}(v),\gamma_1^{\ii}(v),\gamma_2^{\out}(v),\gamma_2^{\ii}(v)].
  \end{equation*}
  
\end{lemma}
\par
\begin{proof}
  Consider an edge $e$ having weight $r$.  Since the loops share no edges, one of the $\gamma_i$ does not
  pass through $e$.  We then have $\langle \lambda(\gamma_i), r\rangle = 0$, so
  $$\langle\lambda(\gamma_1)\otimes\lambda(\gamma_2), r\otimes r\rangle = 0.$$  It follows that
  $\lambda(\gamma_1)\otimes\lambda(\gamma_2)\in N(\mathcal{S})$.

  We consider now $\gamma_1$ and $\gamma_2$ as loops on a stable curve
  $X$ in $\mathcal{S}$.  There is a unique stable form $\omega$ on $X$
  with poles of residue $-1$ at each $\gamma_1^\out(v)$ and residue
  $1$ at each $\gamma_1^\ii(v)$.  By definition, $\Psi(a) =
  e^{\int_{\gamma_2} \omega}$.

  Fix a component of $X$ corresponding to a vertex $v$, which we identify with the Riemann sphere
  punctured at finitely many points.  Normalizing so that $\gamma_1^\ii(v) = 0$ and
  $\gamma_2^\out(v) = \infty$, we have $\omega = \frac{dz}{z}$ on this component and the
  corresponding term of $\int_{\gamma_2}\omega$ is
  \begin{equation*}
    \int_{\gamma_2^\out(v)}^{\gamma_2^\ii(v)}\frac{dz}{z} = \log\frac{\gamma_2^\ii(v)}{\gamma_2^\out(v)} = \log
    [\gamma_1^{\out}(v),\gamma_1^{\ii}(v),\gamma_2^{\ii}(v),\gamma_2^{\out}(v)]^{-1}. \qedhere
  \end{equation*}
\end{proof}

\section{Properties of the RM-tori}
\label{sec:RM-tori}

We now study in more detail the tori $T_\mathcal{S}$ and $A_\mathcal{S}$ introduced in
\S\,\ref{sec:boundary}, computing their dimension, as well of the dimension of the intersection of
$T_\mathcal{S}$ with various subtori of $A_\mathcal{S}$.

\begin{theorem}
  \label{thm:dimension_N}
  Consider a $F$-weighted boundary stratum $\mathcal{S}$ having genus $g$, among whose weights are
  exactly $n$ distinct weights (up to sign) $r_1,\ldots,r_n$. Then the elements $r_1\otimes r_1, \ldots,
  r_n\otimes r_n$ of $\Sym_\ratls(F)$ are linearly independent over $\ratls$.

  Equivalently, $N(\mathcal{S})$ and $A_\mathcal{S}$ have dimension $g(g+1)/2 -n$.
\end{theorem}

\begin{proof}
  The equivalence of these statements is clear from the definition of $N(\mathcal{S})$ and
  $A_\mathcal{S}$.  Suppose first that $\mathcal{S}$ is an irreducible stratum (parameterizing
  irreducible stable curves).  There are $g$ distinct weights $r_1,\ldots,r_g$ which form a basis of
  $F$ over $\ratls$, and the $r_i\otimes r_i$ are then linearly independent in $\Sym_\ratls(F)$ (as is true for
  any basis of a vector space).

  We now show that if the claim holds for some weighted stratum $\mathcal{S}$, then it holds for any
  degeneration $\mathcal{S}'$ of $\mathcal{S}$ obtained by pinching a single curve.  The claim then
  follows for all strata by induction.

  If the new node of $\mathcal{S}'$ has the same weight as some node of $\mathcal{S}$, then
  $N(\mathcal{S}) = N(\mathcal{S}')$, and we are done. Now suppose the new node has distinct weight.
  To finish the proof, we must find an element of $N(\mathcal{S})$ which does not belong to
  $N(\mathcal{S}')$.  In the weighted graph $\Gamma(\mathcal{S}')$, let $e$ be the edge
  corresponding to the new node.  Using the interpretation of pairs of loops on $\Gamma(\mathcal{S}')$ as
  elements of $\bS_\ratls(F)$ from \S\ref{sec:boundary}, it suffices to find a pair of loops on
  $\Gamma(\mathcal{S}')$ which both contain the edge $e$ and have no other edges in common.

  Let $G$ be the graph obtained by deleting $e$ from $\Gamma(\mathcal{S}')$, and let $p$ and $q$ be
  the distinct vertices of $G$ which were joined by $e$.  Since the weight of $e$ is distinct from
  the weights of the other edges of $\Gamma(\mathcal{S}')$, there is no edge of
  $\Gamma(\mathcal{S}')$, which jointly with $e$ separates $\Gamma(\mathcal{S}')$.  Thus $G$ is not
  separated by any of its edges.  It then follows from Menger's theorem (see \cite{bondymurty}) that there
  are two edge-disjoint paths on $G$ joining $p$ to $q$.  These paths yield the required pair of
  loops on $\Gamma(\mathcal{S}')$.
\end{proof}

\begin{cor}
  \label{cor:CisnotN}
  For any $F$-weighted boundary stratum $\mathcal{S}$, the cone $C(\mathcal{S})\subset\bS_\ratls(F)$
  strictly contains the subspace $N(\mathcal{S})$.
\end{cor}

\begin{proof}
  Let $r_i$ be the weights of $\mathcal{S}$.  As the $r_i\otimes r_i$ are linearly independent in
  $\Sym_\ratls(F)$, we may find some $t\in \bS_\ratls(F)$ which pairs positively with $r_1\otimes
  r_1$ and trivially with the other $r_i\otimes r_i$.  This $t$ lies in $C(\mathcal{S})$ but not
  $N(\mathcal{S})$. 
\end{proof}

We now turn to the dimension of the RM-torus $T_\mathcal{S}$.

\begin{lemma}
  \label{lem:dimLambdaone}
  The subspace $\Lambda^1\subset\Sym_\ratls(F)$ has dimension $g$.
\end{lemma}

\begin{proof}
  Under the identification of $F\otimes_\ratls F$ with $\Hom_\ratls(F,F)$ induced by the trace
  pairing, $\Lambda^1$ corresponds to $\Hom_F(F,F)$.
\end{proof}

\begin{prop}
  \label{prop:RMtorusdimension}
  For any admissible $F$-weighted boundary stratum $\mathcal{S}$ of genus $g$, the RM-torus
  $T_\mathcal{S}$ has dimension at most $g-1$.
\end{prop}

\begin{proof}
  From the definition of $T_\mathcal{S}$, we have
  \begin{equation*}
    \dim T_\mathcal{S} = \dim N(\mathcal{S}) - \dim (N(\mathcal{S}) \cap \Ann(\Lambda^1)).
  \end{equation*}
  Under the quotient map $\bS_\ratls(F) \to \bS_\ratls(F) / N(\mathcal{S})$, the images of
  $\Ann(\Lambda^1)$ and $C(\mathcal{S})$ have trivial intersection by the admissibility of
  $\mathcal{S}$.  By Corollary~\ref{cor:CisnotN}, the image of $C(\mathcal{S})$ is nontrivial.  It
  follows that the image of $\Ann(\Lambda^1)$ is not all of $\bS_\ratls(F) / N(\mathcal{S})$.
  Equivalently,
  \begin{equation*}
    \dim \Ann(\Lambda^1) - \dim (\Ann(\Lambda^1) \cap N(\mathcal{S})) < \dim \bS_\ratls(F) - \dim N(\mathcal{S}).
  \end{equation*}
  As $\Ann(\Lambda^1)$ has codimension $g$ in $\bS_\ratls(F)$ by Lemma~\ref{lem:dimLambdaone}, the
  desired inequality follows.
\end{proof}

There are examples where $T_\mathcal{S}$ has dimension less than $g-1$; see \cite[Appendix~A]{bainmoel}.

Choose a basis $r_1, \ldots, r_g$ of $F$, and let $s_1, \ldots, s_g$ be the dual basis with respect
to the trace pairing.  We define
\begin{equation*}
  \epsilon = \sum_{i=1}^g r_i\otimes s_i \in F \otimes_\ratls F.
\end{equation*}

\begin{prop}
  \label{prop:epsilonstuff}
  The element $\epsilon$ lies in $\Lambda^1$ and does not depend on the choice of the basis of $F$.
  Thus we have
  $
    \Lambda^1 = \{ x\epsilon : x\in F\}.
  $
  For every $x\in F$ and $s\otimes t \in \bS_\ratls(F)$, we have the pairing
  \begin{equation}
    \label{eq:1}
    \langle x\epsilon, s \otimes t\rangle = \Tr^F_\ratls(xst)
  \end{equation}
\end{prop}

\begin{proof}
  See \cite[Lemma~6.2]{bainmoel}.  This lemma only calculates the pairing $\langle x\epsilon,
  t\otimes t\rangle$, but the proof of the more general statement is identical.
\end{proof}

We define the evaluation map $\ev\colon \bS_\ratls(F)\to F$ by $\ev(s\otimes t) = st$.

\begin{cor}
  \label{cor:AnnisKer}
  The annihilator $\Ann(\Lambda^1)$ is the kernel of $\ev$.
\end{cor}

\begin{proof}
  If $\alpha\in \Ann(\Lambda^1)$, then we have by Proposition~\ref{prop:epsilonstuff} that
  \begin{equation*}
    0 = \langle x\epsilon, \alpha \rangle = \Tr(\ev(\alpha)x)
  \end{equation*}
  for all $x\in F$.  Since the trace pairing is nondegenerate, it follows that $\ev(\alpha) = 0$.
\end{proof}

By the definition of $A_\mathcal{S}$, we have the identification $\chi(A_\mathcal{S})\otimes\ratls =
N(\mathcal{S})$, where we write $\chi(T)$ for the character group of any torus $T$.  Given any subtorus
$U\subset A_\mathcal{S}$, we write $\Ann(U)\subset N(\mathcal{S})$ for the subspace of characters
which annihilate $U$.  This is a bijection between dimension $d$ subtori of $A_\mathcal{S}$ and
codimension $d$ subspaces of $N(\mathcal{S})$.  With this notation, $\Ann(T_\mathcal{S}) = \Ann(\Lambda^1)$.

\begin{prop} \label{prop:toruschar2}
  For any subtorus $U\subset A_\mathcal{S}$, the intersection of $U$ with the RM-torus $T_\cS$ has
  codimension in $T_\cS$ equal to $\dim \ev(\Ann(U))$, that is
  \begin{equation}
    \label{eq:4}
    \dim(T_\cS)- \dim (U\cap T_\cS) =  \dim \ev(\Ann(U)).
  \end{equation}
  In particular,
  $$\dim(T_\mathcal{S}) = \dim \ev(N(\mathcal{S})).$$
\end{prop}

\begin{proof}
  By Corollary~\ref{cor:AnnisKer}, we have
  \begin{align*}
    \dim \ev(\Ann(U)) &= \dim\Ann(U) - \dim(\Ann(U)\cap\Ann(T_\mathcal{S}))\\
    &= \dim(\Ann(U) + \Ann(T_\mathcal{S})) - \dim\Ann(T_\mathcal{S}).
  \end{align*}
  Also note that $\Ann(U\cap T_\mathcal{S}) = \Ann(U) + \Ann(T_\mathcal{S})$.
  It follows that
  \begin{align*}
    \dim\ev(\Ann(U)) 
    &= \dim\Ann(U\cap T_\mathcal{S}) - \dim\Ann(T_\mathcal{S}) \\
    &= \dim(T_\mathcal{S}) - \dim(U\cap T_\mathcal{S}).
  \end{align*}

  To obtain the last statement, apply \eqref{eq:4} for $U$ the trivial torus.
\end{proof}

As an application, we can now calculate the dimension of $T_\mathcal{S}$ for many strata
$\mathcal{S}$.

\begin{prop}
  \label{prop:dimTSwithloop}
  Suppose that $\mathcal{S}$ is an admissible weighted boundary stratum for which the dual graph
  $\Gamma(\mathcal{S})$ contains an edge joining a vertex to itself.  Then $\dim (T_\mathcal{S}) = 
  g-1$.
\end{prop}

\begin{proof}
  By Propositions~\ref{prop:RMtorusdimension} and \ref{prop:toruschar2} we need only to show that
  $\dim \ev(N(\mathcal{S})) \geq g-1$.  Let $\gamma_1$ be a loop which joins a vertex of
  $\Gamma(\mathcal{S})$ to itself, and let $K \subset F$ be the span of $\lambda(\gamma)$ over all
  loops $\gamma$ which are edge-disjoint from $\gamma_1$.  Since $\lambda$ induces an isomorphism
  $H_1(\Gamma(\mathcal{S}); \ratls)\to F$, and deleting a loop from a graph reduces the rank of its
  homology by one, the dimension of $K$ is $g-1$.  By Lemma~\ref{le:NSloops},
  $\lambda(\gamma_1)\otimes K\in N(\mathcal{S})$, so $\lambda(\gamma_1)K\subset \ev(N(\mathcal{S}))$.
\end{proof}

\section{Reduction to the boundary condition}
\label{sec:redtobd}

In this section we show how the main theorem reduces to a containment statement
about tori in the locus of stable forms for some boundary strata of $\barmoduli[4]$.
For this purpose we call a boundary stratum $\cS$ of 
$\moduli[4]$  {\em relevant}, if it parametrizes curves of geometric genus zero
and if $\dim(\cS) \geq 3$. Since for geometric genus zero curves  the 
dimension of $\cS$ equals six minus the number of irreducible components 
of any stable curve of the stratum, the last condition is
equivalent to having at most three irreducible components.
\par
\begin{proof}[Proof of Theorem~\ref{thm:Intro}]
  Suppose that, contrary to the claim of the theorem, for some order $\Ord$ and some symplectic
  $\Ord$-module $M$ the component $\RA[{\Ord,M}]$ of the real multiplication locus $\RA[\Ord]$ is
  generically contained in $\moduli[4]$. We denote by $\RM[{\Ord,M}]$ the preimage of this component
  under the Torelli map.  Since the Hilbert modular variety is not compact, the intersection
  $\partial\RM[{\Ord,M}]$ of $\overline{\RM[{\Ord,M}]}$ and the boundary part of the boundary of
  $\barmoduli$ consisting of curves with non-compact Jacobian inside $\barmoduli$ is non-empty.  In
  fact $\partial\RM[{\Ord,M}]$ must be a divisor on $\overline{\RM[{\Ord,M}]}$, hence all
  irreducible components of $\partial\RM[{\Ord,M}]$ are of dimension three.  By \cite[Corollary
  5.6]{bainmoel}, $\partial\RM[{\Ord,M}]$ lies in the union of boundary strata parameterizing curves
  of geometric genus zero.  More precisely, by Theorem~\ref{thm:RMboundary} the boundary of
  $\RM[{\Ord,M}]$ lies in the image of the $\RS[E]$ for some extension classes $E$.
\par
All together, each irreducible component of $\partial \RM[{\Ord,M}]$
generically lies in some relevant admissible weighted boundary stratum and 
for each the relevant admissible weighted boundary stratum $\cS$ that 
$\partial \RM[{\Ord,M}]$ intersects, there is some extension class $E$
such that $\RS[E]$ is of dimension three.
\par
The following Propositions~\ref{prop:cortorelli}, 
\ref{prop:bdcondition} and \ref{prop:redtononsep}
provide the contradiction we need to prove Theorem~\ref{thm:Intro} for
all the topological types of $\cS$.
\end{proof}
\par
\begin{prop} \label{prop:cortorelli}
For each relevant weighted boundary stratum $\cS$ of $\barmoduli[4]$  
without separating nodes, the topological type of $\cS$ being
listed in Figure~\ref{fig:graphs}, the cross-ratio map $\CR$  is
finite. In particular for each extension class $E$,
$$\dim \RS[E] = \dim (\CR(\cS) \cap T_{\cS,E}),$$
where the intersection is taken inside the ambient torus $A_\cS$.
\end{prop}
\par
\begin{prop} \label{prop:bdcondition}
For each relevant weighted boundary stratum $\cS$ of $\barmoduli[4]$  
without separating nodes, the topological type of $\cS$ being
listed in Figure~\ref{fig:graphs}, 
the intersection of $\CR(\cS)$ with each translated cross-ratio
torus $T_{\mathcal{S}, E}$ inside the ambient torus is of dimension at most two.
\end{prop}

It remains to show that a component of $\RA$ contained in $\moduli[4]$ cannot only meet boundary
strata which have separating curves.

\begin{prop} \label{prop:redtononsep}
Suppose that a component $\RA[{\Ord,M}]$ of the real
multiplication locus is generically contained in the
Torelli image of $\barmoduli[4]$. Suppose moreover, 
that the closure of $\RM[{\Ord,M}]$ intersects the image in $\barmoduli[4]$ 
of a weighted relevant boundary 
stratum $\cS$ parameterizing stable curves with a separating node.
Then there exists also an irreducible component of $\partial\RM[{\Ord,M}]$
contained in a relevant boundary stratum  parameterizing stable curves without 
a separating node.
\end{prop}
\par
Proposition~\ref{prop:cortorelli} is a weaker version of 
Theorem~\ref{thm:weaktorelli}.
The proof of the other two propositions will occupy the rest of the paper.

\paragraph{Relevant boundary strata of $\moduli[4]$.}

Figure~\ref{fig:graphs} contains the complete list of relevant boundary 
strata of $\moduli[4]$ parameterizing stable curves without 
a separating node. We will refer to the stratum $(x,y)$ as the
stratum in row $x$ and column $y$. The strata $(2,2)$ (``the
$[5] \times^5 [5]$-stratum'') and $(4,2)$ (``the doubled triangle'')
will need a special treatment below.
\par
\begin{figure}[htbp]
  \centering
  \includegraphics{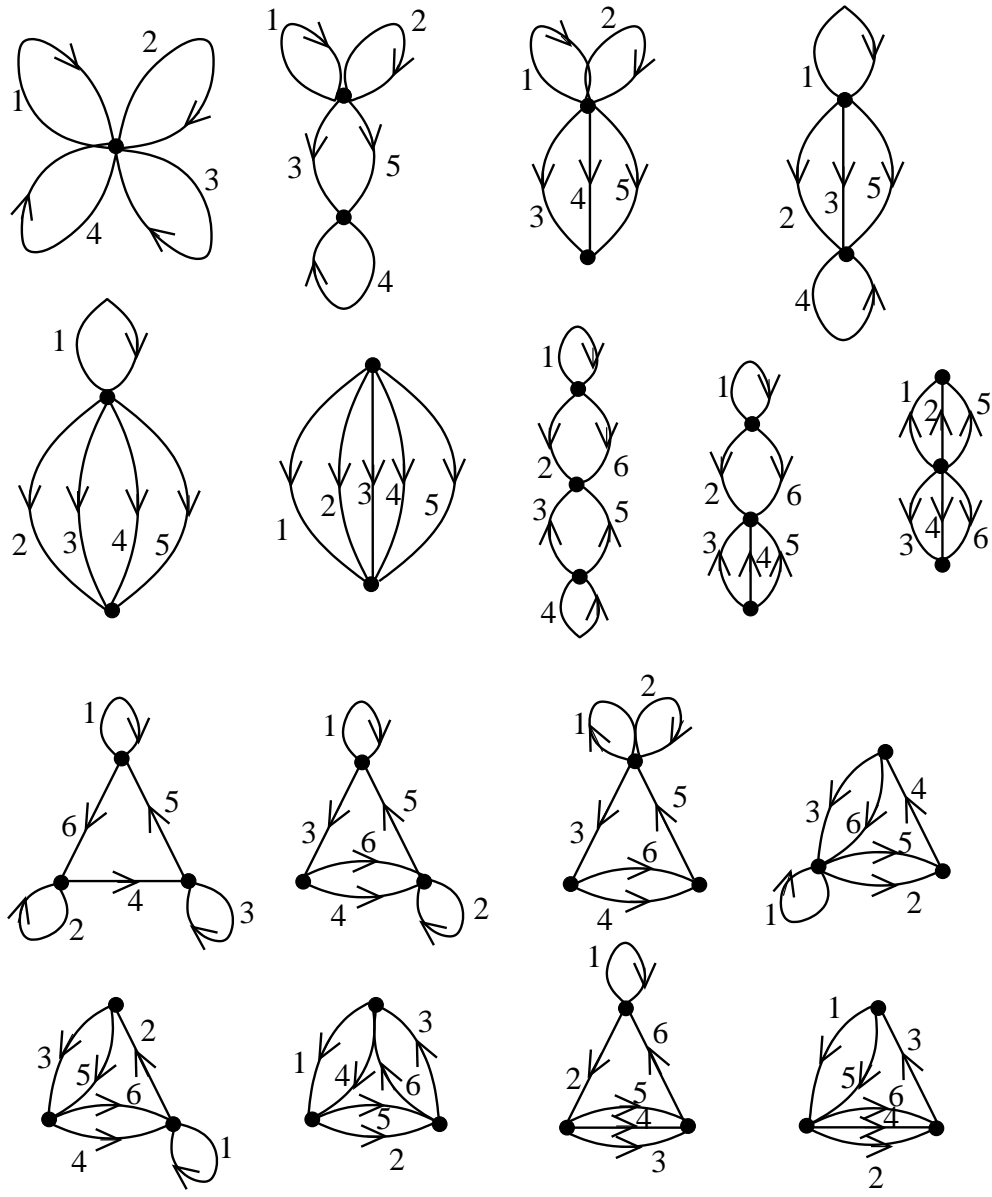}
  \caption{Relevant genus four stable curves without separating nodes}
  \label{fig:graphs}
\end{figure}
\par
The arrows are chosen arbitrarily. Their purpose is to label
the marked points on the normalization of the stable curve that
are glued together. Our convention is that on the edge with label
$k$ the points $P_k$ and $Q_k$ are glued together, where 
the point $P_k$ sits on the outgoing component and $Q_k$ sits
on the incoming component.  If the graph is given an $F$-weighting, we call $r_k$ the weight of the
$k$th edge.
\par
There are many choices for labelling of the edges. 
Our choices have the property that the first four weights always span $F$:
\par
\begin{lemma} \label{le:resbasis}
  Consider an $F$-weighted graph $\Gamma$ containing edges $e_1, \ldots, e_4$ having weights $w_1, \ldots,
  w_4$.  Then the weights $w_i$ span $F$ if and only if the complement of the $e_i$ in $\Gamma$ is a tree.
  \par
  In particular, with the choice of labelling given in Figure~\ref{fig:graphs}, every relevant $F$-weighted
  boundary stratum with no separating nodes has the property that the weights $\{r_1, r_2, r_3,
  r_4\}$ are a $\ratls$-basis of $F$.
\end{lemma}
\par
\begin{proof}
  Recall that there is the isomorphism $\lambda\colon H_1(\Gamma; \ratls)\to F$.  Let $A\subset F$
  be the annihilator of the span of the $w_i$ with respect to the trace pairing.  If the complement
  of the $e_i$ contains a loop $\gamma$, then by the definition of $\lambda$, the nonzero element
  $\lambda(\gamma)$ pairs trivially with each $w_i$.  Thus $A$ is nontrivial, and the $w_i$ do not
  span.  The converse follows similarly.
\end{proof}
\par
Consequently, we may set $(s_1,s_2,s_3,s_4)$ to be the basis of $F$ dual to
the basis
$(r_1,r_2,r_3,r_4)$, and we keep this notation throughout the rest of the paper.

\section{Cross-ratios for nice boundary strata without separating nodes}
\label{sec:CRnice}

The next theorem is completely analogous to \cite[Proposition~8.3]{bainmoel}, 
except that we now work in genus~4 instead of genus~3. It is a
simple Torelli type theorem for  relevant boundary strata without separating nodes.
\par
For any weighted boundary stratum $\cS$ let $\phi_\iota$ be the involution which changes each weight
to its negative. We say that $\cS$ is $\phi_\iota$-invariant, if $\phi_\iota(\cS)$ and $\cS$ are
topologically equivalent, in which case $\phi_\iota$ restricts to an involution of $\mathcal{S}$.
Among the relevant boundary strata without separating nodes, the $\phi_\iota$-invariant strata are
precisely the strata $(1,1)$, $(1,2)$, $(2,2)$, $(2,3)$, and $(3,3)$.  Note that whether or not a
weighted stratum is $\phi_\iota$-invariant depends only on its topological type and not on the
choice of weights.
\par

By the definition of $\phi_\iota$, each period matrix entry $\Psi(x)$ is equivariant with respect to
the involution $\phi_\iota$.  If $\mathcal{S}$ is $\phi_\iota$-invariant, we define $\mathcal{S}'$
to be the quotient of $\mathcal{S}$ by $\phi_\iota$.

Recall from \S\,\ref{sec:boundary} that we have a canonical morphism $\CR\colon\mathcal{S}\to
A_\mathcal{S}$.  A basis $\tau_1, \ldots, \tau_n$  of $N(\mathcal{S})$ determines an isomorphism of
$A_\mathcal{S}$ with $(\cx^*)^n$, and in these coordinates, $\CR$ is simply the product of the
functions $\Psi(\tau_i)$.  By the above discussion, $\CR$ factors through $\phi_\iota$ to define a
morphism $\CR\colon \mathcal{S}'\to A_\mathcal{S}$.
\par
\begin{theorem} \label{thm:weaktorelli} Given a relevant stratum without separating nodes $\mathcal{S}$,
  the morphism $\CR$ is an embedding of $\mathcal{S}$ (or $\mathcal{S}'$ if $\mathcal{S}$ is
  $\phi_\iota$-invariant) in $A_\mathcal{S}$. 
\end{theorem}
\par
\begin{proof}
  We give the details for some strata where essential arguments show up and leave the remaining
  verifications to the reader.
  \par
  We first consider the irreducible stratum $(1,1)$.  This stratum is $\phi_\iota$-invariant, and
  the involution has the effect of swapping each pair $P_i, Q_i$.  Here $s_i\otimes s_j \in N(\cS)$
  for all $i \neq j$ and we abbreviate $\Psi(s_i\otimes s_j) = R_{ij}$ where
  $$ R_{ij} = [P_i,Q_i,P_j, Q_j]$$ 
  by Lemma~\ref{le:NSloops} or \cite[Proposition~8.3]{bainmoel}.  We fix $P_1=0$, $Q_1 = \infty$ and
  $P_2 = 1$.  We then have $R_{12} = Q_2$ and $P_3 R_{13} = Q_3$.  Given $R_{23}$, we need to solve
  a quadratic equation to recover $P_3$ and $Q_3$.  Similarly, given $R_{14}$ and $R_{24}$ we need
  to solve a quadratic equation to recover $P_4$ and $Q_4$.  Thus there are four possibilities for
  the tuple $(P_3, Q_3, P_4, Q_4)$.  The cross-ratio $R_{34}$ eliminates two of these solutions, and the
  remaining two are related by the involution $\phi_\iota$.
  \par
  Next we consider the ``$[5] \times ^5 [5]$''-stratum $(2,2)$.  Again, $\phi_\iota$ swaps each pair
  $P_i, Q_i$.  Here $s_i\otimes (s_j - s_k) \in
  N(\cS)$ for all distinct $i,j,k$.  We normalize
  $$P_1=Q_1 = 1, \quad P_2=Q_2=0, \quad P_5=Q_5 = \infty.$$
  Then 
  $$\begin{array}{ll} 
    \Psi(s_1 \otimes (s_4-s_2)) = (1-P_4)(1-Q_4), 
    &  \Psi(s_2 \otimes (s_4-s_1))  =P_4 Q_4 \\
    \Psi(s_1 \otimes (s_3-s_2)) =   (1-P_3)(1-Q_3), 
    &  \Psi(s_2 \otimes (s_3-s_1)) =P_3 Q_3. \\
  \end{array}
  $$
  The first line determines two possibilities for $P_4$ and $Q_4$, and the second line determines
  two possibilities for $P_3$ and $Q_3$. Using $\Psi(s_3\otimes(s_1-s_2))$ eliminates two of the four
  possibilities for $(P_3, P_4, Q_3, Q_4)$, with the remaining two possibilities related by
  $\phi_\iota$. 
  \par
  As typical examples for the remaining cases we take the stratum $(1,4)$.  Here $N(\cS)$ is
  generated by $s_1\otimes s_2$, $s_1 \otimes s_3$, $s_4 \otimes s_2$ and $s_4 \otimes
  s_3$. Normalizing
  $$P_1=P_4=0, \quad Q_1=Q_4=\infty, \quad P_2=Q_2=1 $$ 
  we obtain
  $$\begin{array}{ll}
    \Psi(-s_1 \otimes s_2) = Q_5, 
    &  \Psi(-s_1 \otimes s_3) = Q_3, \\
    \Psi(-s_4 \otimes s_2) = P_5,
    &  \Psi(-s_4 \otimes s_3) = P_3, \\
  \end{array}
  $$
  so these four cross-ratios determine the remaining four points.
\end{proof}
\par

\paragraph{Gerritzen's equation.}

Given the above Torelli theorem, it is natural to ask what the
image of $\CR$ in the ambient torus is. For the irreducible 
stratum $(1,1)$ the question of finding the equation
cutting out the image of $\CR$ has been solved by \cite{gerritzen}.
For all but one exceptional stratum, we will be able to avoid the use of this equation.
For one exceptional stratum we indeed need to determine the image
of $\CR$ and this equation can be obtained as a limit of Gerritzen's equation.
\par
\begin{prop}[{\cite[Proposition~4.3.1]{gerritzen}}] \label{prop:gerritzeneq}
For the irreducible stratum, the image of $\CR$ in the ambient torus $A_{\cS}$ 
is given, in the coordinates introduced in the proof of 
Theorem~\ref{thm:weaktorelli}, as the vanishing locus of the 
function $F = \Delta H - G$, where
\begin{equation*}
\begin{aligned}
\Delta &= (R_{12}-1)(R_{13}-1)(R_{14}-1)(R_{23}-1)(R_{24}-1)(R_{34}-1) \\
H &=
R_{12}R_{13}R_{14}R_{23}R_{24}R_{34}-R_{12}R_{14}R_{24}-R_{13}R_{14}
R_{34}-R_{23}R_{24}R_{34} \\
&- R_{12}R_{13}R_{23} + R_{14}R_{23} + R_{13}R_{24} + R_{12}R_{34} \\ 
 G & = R_{12}R_{34}(R_{13}-1)^2(R_{14}-1)^2(R_{23}-1)^2(R_{24}-1)^2 \\
& +  R_{13}R_{24}(R_{12}-1)^2(R_{14}-1)^2(R_{23}-1)^2(R_{34}-1)^2 \\
&+  R_{14}R_{23}(R_{12}-1)^2(R_{13}-1)^2(R_{24}-1)^2(R_{34}-1)^2. \\
\end{aligned}
\end{equation*}
\end{prop}
\par
The validity of the equation, hence the fact that we use the
same conventions on cross-ratios as Gerritzen, can be checked
by plugging in the definition of the cross-ratios.

\paragraph{A hypothetical sufficiency theorem.}

In \cite{bainmoel}, we proved that the necessary condition of Theorem~\ref{thm:RMboundary} is also
sufficient in genus three.  The proof relied heavily on the fact that the Schottky problem is
trivial in genus three, that is, the image of the Torelli map $\moduli\to \A$ is dense.  In higher
genus, we do not know whether our condition is also sufficient; however, under the assumption that
a component $\RA[\mathcal{O}, M]$ of the real multiplication is contained in the Schottky locus
sufficiency holds at least in some cases.

We say that a stable curve is \emph{nice} if the complement of any two nodes is connected (such
curves are sometimes called three-connected).  A
boundary stratum is nice if it consists of nice stable curves.

\begin{theorem}
  \label{thm:suff}
  Assuming that a component $\RA[\mathcal{O}, M]$ is contained in the closure of $t(\moduli[4])$, the
  necessary condition of Theorem~\ref{thm:RMboundary} for  geometric genus zero stable curve to lie
  in the boundary of $\RM[\mathcal{O}, M]$ is also sufficient for nice stable curves.
\end{theorem}

We emphasize that this theorem has no use outside of this paper, as we are proving that the
hypothesis of the theorem never holds.  As the proof of Theorem~\ref{thm:suff} is essentially the
same as in \cite{bainmoel}, we only sketch the idea of the proof here.

Consider a nice boundary stratum $\mathcal{S}$ and choose a basis $\tau_1, \ldots, \tau_n$ of
$N(\mathcal{S})$.  By Theorem~\ref{thm:weaktorelli}, the map $\CR\colon\mathcal{S}\to(\cx^*)^n$ is
either two-to-one or biholomorphic onto its image $\CR(\mathcal{S})$.  In either case this map is
open.  In genus three, this map is open as well and is also onto.  In \cite{bainmoel}, we extended
this to an open map $\Xi\colon U\to \cx^m\times(\cx^*)^n$ (for some neighborhood $U$ of
$\mathcal{S}$) sending $\mathcal{S}$ to $\{\bzero\}\times(\cx^*)^n$.  In genus four, the map $\Xi$
is defined in the same way and is also an open map
$\Xi\colon U\to\cx^m\times\CR(\mathcal{S})\subset\cx^m\times(\cx^*)^n$.  In order to define the map
$\Xi$, it is necessary for $\mathcal{S}$ to be nice.

The map $\Xi$ sends a component of the real multiplication locus to a subvariety
$T\subset(\cx^*)^m\times(\cx^*)^n$ which is a translate of a torus.  As $\Xi$ is open, to show that
a point $p\in\mathcal{S}$ is in the boundary of the real multiplication locus, it suffices to show
that $\Xi(p)$ is in the boundary of $T$.  Thus the problem is reduced to calculating the boundary of
a torus.  In \cite[Theorem~8.14]{bainmoel}, we construct an explicit one-dimensional torus
$T_1\subset T$ which limits on $\Xi(p)$, showing that $\Xi(p)$ is indeed in the boundary of $T$.

In genus three, this proof uses in an essential way that a generic Abelian variety is a Jacobian.
In higher genus, this proof breaks down, since we don't know whether the torus $T_1$ is contained in
the Schottky locus.  However, assuming that the component of the real multiplication we are
considering is contained in the Schottky locus, this is automatic and the proof carries through.

\section{Checking nice boundary strata without separating nodes}
\label{sec:check}

In this section we develop two criteria on the dual graphs
of strata to test whether Proposition~\ref{prop:bdcondition} holds.  These criteria will apply to
each of the strata in Figure~\ref{fig:graphs} except for two exceptional strata which we handle with
{\it ad hoc} arguments.

In the following, it will be useful to encode loops on such a dual graph by the labeling of the
edges.  We use the convention that the first digit corresponds to the edge used first and an
overline corresponds to using the edge in the direction pointing opposite the arrow. E.g.\ in the
stratum $(1,2)$, the loop $(3\bar{5})$ turns the counterclockwise around the middle circle, starting
at the upper vertex.
\par
\paragraph{The disjoint loop argument.}

The first criterion rules out strata whose dual graphs contain disjoint loops:
\begin{prop}
Let $\Gamma(\cS)$ be the dual graph of a relevant nice boundary stratum 
without separating nodes $\cS$. Suppose $\Gamma$ contains
two vertex-disjoint simple loops $\gamma_1$ and $\gamma_2$. 
Then Proposition~\ref{prop:bdcondition} holds for $\cS$.
\end{prop}
\par
\begin{proof}
Suppose the contrary holds, i.e.\ $T_{\cS,E} \subset \overline{\CR(\cS)}$. Let 
$a= \lambda(\gamma_1) \otimes \lambda(\gamma_2) \in N(\cS)$ as in Lemma~\ref{le:NSloops}.
We claim that $\ev(a) \neq 0$. To justify this, it suffices to show
that the field element associated with any simple loop is non-zero.
This is a consequence of the isomorphism $\lambda: H_1(\Gamma(\cS), \zed) \to 
{\mathcal I}^\vee$ and the fact that such a loop is non-zero in $H_1(\Gamma(\cS, \zed))$.
\par
Since the loops are vertex-disjoint, the Lemma~\ref{le:NSloops} implies $\Psi(a) \equiv 1$, 
i.e.\ $T_{\cS,E}$ contained in the torus $U$ with $\Ann(U) = \langle a \rangle$. 
Together with $\ev(a) \neq 0$ this contradicts Proposition~\ref{prop:toruschar2}.
\end{proof}

\paragraph{The shared vertex argument.} Recall that a graph is called \emph{$n$-connected} if it can not be
disconnected by removing $n-1$ edges.  We say that it is \emph{precisely $n$-connected} if it is
$n$-connected and can be disconnected by removing $n$ edges.  Note that the disjoint
loop argument applies to any precisely $2$-connected stratum.

\begin{prop}
  Let $\Gamma(\cS)$ be the dual graph of a relevant nice boundary stratum without separating nodes
  $\cS$.  Suppose $\Gamma(\cS)$ contains two edge-disjoint loops $\gamma_1$ and $\gamma_2$ having
  exactly one vertex $v$ in common.  Suppose moreover, that there is some precisely $2$-connected graph
  $\Gamma'$ obtained from $\Gamma(\cS)$ by gluing an edge $e$ into $v$ such that $\gamma_1$ and
  $\gamma_2$ yield vertex-disjoint loops in $\Gamma'$, and moreover for all such graphs $\Gamma'$
  the following condition holds.  There is a loop $\gamma_3$ on $\Gamma'$ such that $\gamma_3$ and
  $\gamma_1$ or $\gamma_3$ and $\gamma_2$ are vertex-disjoint.
\end{prop}
\par
\begin{proof}
  We let $\tilde{\mathcal{S}}$ be the partial Deligne-Mumford compactification of $\mathcal{S}$
  obtained by adjoining any weighted stable curve which has the same set of weights as the curves
  parametrized by $\mathcal{S}$.  For each degeneration $\mathcal{S}'$ in
  $\tilde{\mathcal{S}}\setminus\mathcal{S}$, we then have $N(\mathcal{S'}) = N(\mathcal{S})$, so the
  morphism $\CR$ extends to a morphism $\CR\colon \tilde{\mathcal{S}}\to A_\mathcal{S}$.

  Now suppose the contrary holds, i.e.\ $T_{\cS,E} \subset \overline{\CR(\cS)}$.  Let $a =
  \lambda(\gamma_1) \otimes \lambda(\gamma_2) \in N(\cS)$ and consider the intersection of
  $\CR(\tilde{\mathcal{S}})$ with the subtorus $U$ given by $\Ann(U) = \langle a \rangle$.  This
  intersection is nonempty and consists of the union of all degenerations $\CR(\mathcal{S}')$ which
  correspond to some graph $\Gamma'$ as in the statement.  By Proposition~\ref{prop:toruschar2}, the
  intersection $U\cap T_{\mathcal{S}, E}$ is a translate of a two-dimensional subtorus of
  $A_\mathcal{S}$, thus it is contained in one degeneration $\CR(\mathcal{S}')$ of
  $\CR(\mathcal{S})$ as above.  In what follows we fix this degeneration $\mathcal{S}'$.

  Let $b = \lambda(\gamma_1) \otimes \lambda(\gamma_3)$ resp.\ $\lambda(\gamma_3) \otimes
  \lambda(\gamma_2)$ depending on which loops are disjoint on $\cS'$. The preceding argument
  together with Lemma~\ref{le:NSloops} imply that on $\CR(\cS') \cap U$ the function $\Psi(b)$ is
  identically one.  Consider the torus $U_2$ defined by $\Ann(U_2) = \langle a, b \rangle$.  Since
  $\Psi(b) \equiv 1$ on $\CR(\mathcal{S}')$, we have $U_2\cap T_{\mathcal{S}, E} = U\cap
  T_{\mathcal{S}, E}$, so this intersection is two-dimensional.  Since $\dim \ev(\Ann(U_2)) = 2$,
  this contradicts Proposition~\ref{prop:toruschar2}.
\end{proof}
\par
These two arguments allow us to prove Proposition~\ref{prop:bdcondition} in all but
two cases. The disjoint loop argument applies to the $2$-connected strata $(1,2)$,
$(1,4)$, $(2,3)$,  $(2,4)$, $(3,1)$, $(3,2)$, $(3,3)$, $(4,1)$, and $(4,3)$.
\par
To deal with the stratum $(1,1)$ we apply the shared vertex argument to $\gamma_1 = (1)$
and $\gamma_2=(2)$. There is only one precisely $2$-connected degeneration, namely
the stratum $(1,2)$. Obviously $\gamma_3$  with the required properties exists.
\par
To deal with the stratum $(1,3)$ we apply the shared vertex argument to $\gamma_1 = (1)$
and $\gamma_2=(3\bar{4})$.  There are three precisely $2$-connected degenerations
that make $\gamma_1$ and $\gamma_2$ disjoint. The reader will check easily
that in all three cases a loop $\gamma_3$ with the required properties exists.
\par
To deal with the stratum $(2,1)$ we use the loops $\gamma_1 = (1)$ and $\gamma_2=(2\bar{3})$
for the shared vertex argument. To deal with the stratum $(2,5)$ we use the 
loops $\gamma_1 = (1\bar{2})$ and $\gamma_2=(3\bar{4})$.
To deal with the stratum $(3,4)$ we use the 
loops $\gamma_1 = (3\bar{6})$ and $\gamma_2=(2\bar{5})$ for the shared vertex argument.
To deal with the stratum $(4,4)$ we use the loops $\gamma_1 = (1\bar{5})$
and $\gamma_2=(2\bar{6})$. In all these cases, there is only one precisely $2$-connected
degeneration, and the required $\gamma_3$ exists.
\par

\paragraph{The exceptional cases 'doubled triangle' and $[5] \times^5 [5]$.}

Finally, we treat two exceptional cases separately.
\par 
\begin{prop} 
Proposition~\ref{prop:bdcondition} holds for the ``$[5] \times^5 [5]$-stratum'' 
given by the graph $(2,2)$. 
\end{prop}
\par
\begin{proof}
A basis of $N(\cS)$ for this stratum in given by $s_i \otimes s_j - s_3 \otimes s_4$, 
with $i < j$ and $(i,j) \neq (3,4)$. We view this stratum as a degeneration
of the irreducible stratum, obtained by unpinching the node with label $5$.
We derive the equation of the $\CR$-image of this stratum from
the equation of the irreducible stratum in Proposition~\ref{prop:gerritzeneq}.
The coordinates $\tR_{ij} = \Psi(s_i \otimes s_j - s_3 \otimes s_4)$
are related to the coordinates in that proposition by $\tR_{ij} = R_{ij}/R_{34}$. 
Pinching the node with label $5$ takes $R_{34}$ to $\infty$. Hence in order
to determine the image of $\CR$ for the $[5] \times^5 [5]$-stratum, 
we rewrite Gerritzen's equation in terms of the $\tR_{ij}$ and $R_{34}$,
and consider the leading term for $R_{34} \to \infty$. 
\par
Consequently, in these coordinates the image of $\CR$ is the variety $V(F_5)$ cut out by $F_5=0$, where
$$ F_5 = \tR_{12}\tR_{13}\tR_{14}\tR_{23}\tR_{24} -  \tR_{12}\tR_{14}\tR_{23}
- \tR_{12}\tR_{13}\tR_{24} - \tR_{13}\tR_{14}\tR_{23}\tR_{24}.$$ For convenience, we relabel the
coordinates $\tR_{ij}$ as $Z_1, \ldots, Z_5$, using the lexicographical order.

Consider the vectors of exponents $v_1 =
(1,1,1,1,1)$, $v_2=(1,0,1,1,0)$, $v_3=(1,1,0,0,1)$, $v_4=(0,1,1,1,1)$ for the monomials $m_1, \ldots,
m_4$ appearing in $F_5$.  Suppose first that $V(F_5)$ contains a translate of the torus $T$
parametrized by $f_\ba(t) = (t^{a_1}, \ldots. t^{a_5})$.  We have $m_i\circ f(t) = t^{n_i}$, where $n_i
= v_i \cdot \ba$.  It follows that for $T$ to be contained in $V(F_5)$, we must have that for each
$v_i$, there is some other $v_j$ such that $(v_i - v_j)\cdot \ba = 0$.

Now suppose $F_5$ contains a translate of the three-dimensional torus $T_\cS$.  Let $P\subset
\ratls^5$ be the three-dimensional subspace which parmetrizes $T$, and let $N = P^\perp\subset
\ratls^5$.  Given $v_i$, by the above discussion there must be some $j\neq i$ such that $v_i -
v_j\in N$ (using that a vector space over $\ratls$ can not be the union of proper subspaces).  
Since the span of $\{v_1,v_2,v_3,v_4\}$ is three-dimensional, this is only 
possible if there is a basis $n_1,n_2$ of $N$ such that 
$n_1$ is the difference of two of the $v_i$ and $n_2$ is the
difference of the other two $v_i$. Suppose that $n_1 = v_2-v_1$
and $n_2 = v_3 - v_4$, the other two cases will lead to the same contradiction.
\par
By Proposition~\ref{prop:epsilonstuff} the condition that $n_2$ is 
perpendicular to $T_\cS$ is equivalent to 
$$\trace(x (s_1s_2-s_1s_4-s_2s_3+s_3s_4)) = 0 \quad \text{for all} \quad x \in F,$$ 
i.e.\ $0=s_1s_2-s_1s_4-s_2s_3+s_3s_4=(s_1-s_3)(s_2-s_4)$.
This contradicts that the $s_i$ are a $\ratls$-basis of $F$.
\end{proof}
\par
\begin{prop} 
Proposition~\ref{prop:bdcondition} holds for the ``doubled triangle-stratum'' $\cS$
given by the graph $(4,2)$. 
\end{prop}
\par
\begin{proof}
By Lemma~\ref{le:NSloops} the three pairs of loops 
$((1\bar{4}),(3\bar{6}))$, $((2\bar{5}), (1\bar{4}))$
and $((3\bar{6}),(2\bar{5})) $ define elements of $N(\cS)$.
Their $\Psi$-images are 
$$R_1= [P_1,Q_3,Q_6,P_4], \quad R_2= [P_2,Q_1,Q_4,P_5], \quad R_3= [P_3,Q_2,Q_5,P_6].$$ 
\par
Consider now the loops $(123)$ and $(456)$. By the same lemma
they define an element in $N(\cS)$ whose $\Psi$-image is
$$
R_4 = [P_1,Q_6,P_4,Q_3][P_2,Q_4,P_5,Q_1][P_3,Q_5,P_6,Q_2] = 
\frac{1}{1-R_1}\frac{1}{1-R_2}\frac{1}{1-R_3}
$$
\par
Since $\cS$ is irreducible and $\dim \CR(\cS) =3$, if $\CR(\cS)$ contains the 
RM-torus, then $\CR(\cS)$ is equal to that torus. But the above equation 
obviously does not cut out a subtorus of the ambient torus $A_\cS$ with
coordinates $R_1,R_2,R_3,R_4$. 
\end{proof}
\par

\section{Strata with separating curves}
\label{sec:suff}

Finally, we prove Proposition~\ref{prop:redtononsep}, completing the proof of Theorem~\ref{thm:Intro}.

\begin{proof}[Proof of Proposition~\ref{prop:redtononsep}]
By Theorem~\ref{thm:RMboundary} there is an admissible, $\mathcal{I}$-weighted
boundary stratum $\cS'$ such that the boundary component of
$\RM[{\Ord,M}]$ given by hypothesis lies in the image of $\cS'$
under the forgetful map $\cS' \to \barmoduli[4]$. Moreoever, by
hypothesis, the dual weighted graph $\Gamma'=\Gamma(\cS')$ has a separating edge $e$. 
Let $\widetilde{\Gamma}$ be the weighted graph obtained by contracting the
separating edges of $\Gamma'$, preserving weights on the other edges.
If $\widetilde{\Gamma}$ happens to be not nice, say the pair of edges $f$ and $g$ disconnects
$\widetilde{\Gamma}$, then we contract $g$ preserving weights on the other edges. 
The resulting weighted graph will still
be admissible, since the weights on $f$ and $g$ are $r$ and $-r$.
We keep contracting edges until we arrive at a nice graph $\Gamma$. 
If $\widetilde{\Gamma}$ was nice to start with, we take $\Gamma = \widetilde{\Gamma}$.
\par
Let $\cS$ be the corresponding boundary stratum and let $\overline{\cS}$ be the partial
Deligne-Mumford compactification, adding those stable curves whose dual graphs lie between $\Gamma'$
and $\Gamma$.  We want to intersect the cross-ratio images of these spaces with the translated
RM-torus $T_{\cS,E}$. The situation is summarized in the following diagram, where we emphasize that
the cross-ratio map on left is not injective but the one in the middle is injective or two-to-one.
$$\xymatrix{
\cS' \,\, \ar[d]^\CR \ar@{^{(}->}[r] & \overline{\cS} \,\,\ar[d]^\CR \ar[drr]^\CR& & T_{\cS,E} 
\ar@{^{(}->}[d]\\
\CR(\cS') \,\,\ar@{^{(}->}[r] & \CR(\overline{\cS})\,\, \ar@{^{(}->}[rr] & & A_\cS \\
}$$
\par
We may restrict ourselves the case $\dim \cS \geq 4$, since otherwise $\cS'$ cannot contain a
boundary divisor of the real multiplication locus. For the stratum $[5] \times^5 [5]$ there is no
codimension-one degeneration with a separating edge. Since all the other strata $\cS$ with $\dim \cS
\geq 4$ have a loop joining a node to itself, we conclude from Proposition~\ref{prop:dimTSwithloop}
that $\dim T_{\cS,E} = 3$.
\par
The separating edge may split the stable curves into two components
either of genera $2$ and $2$ (the $(2,2)$-case) or of genera $1$ and $3$ (the
$(1,3)$-case). A separating edge also defines a splitting of $F$
into two $\ratls$-subspaces $F_1$ and $F_2$ generated by the
$\lambda$-images of loops in the components of $\Gamma' \setminus \{e\}$.
Each element $a \in F_1 \otimes F_2$ defines by Lemma~\ref{le:NSloops} an element
of $N(\cS)$, and $\CR(\cS')$ is contained in the subtorus defined by $\psi(a)=1$
for all $a \in F_1 \otimes F_2$.
\par
We claim that it is enough to show that $T_{\mathcal{S}, E} \cap \CR(\mathcal{S})\neq\emptyset$.
Suppose that this intersection is in fact nonempty.  By the sufficiency criterion
Theorem~\ref{thm:suff}, this intersection belongs to the intersection of $\overline{\RM[\mathcal{O},
  M]}$ with the boundary of $\barmoduli[4]$.  As $\dim \RM[\mathcal{O}, M]=4$, each irreducible
component of this intersection must be three dimensional.  Thus $T_{\mathcal{S}, E} \cap
\CR(\mathcal{S})$ is contained in a three-dimensional component of $\bdry\RM[\mathcal{O}, M]$ which
lies in some stratum (possibly obtained by further undegenerating $\mathcal{S}$) without separating nodes.
\par
We start with the case of the irreducible stratum $\cS$, hence $\dim A_\cS = 6$.  By the above
discussion, we must show that
the intersection $T_{\mathcal{S}, E} \cap \CR(\overline{\mathcal{S}})$ is not contained in
$\CR(\mathcal{S}')$.  This intersection is at least two-dimsensional, so it suffices to show that
$T_{\mathcal{S}, E} \cap \CR(\mathcal{S}')$ is at most one-dimensional.  In the $(2,2)$-case  
$\dim (F_1 \otimes F_2) = 4$, hence $\CR(\cS') = \CR(\cS') \cap T_{S,E}$.
Proposition~\ref{prop:toruschar2} applied to the torus $U = \CR(\cS')$ and
$\dim \ev(F_1 \otimes F_2) \geq 2$ shows that this intersection is one-dimensional.
In the $(1,3)$-case  $\dim (F_1 \otimes F_2) = 3$ and $\dim \CR(\cS')$
is at least three by the genus three analog of Theorem~\ref{thm:weaktorelli} 
\cite[Corollary 8.4]{bainmoel}, hence $\CR(\cS')$ coincides with the subtorus of $A_\cS$ cut out by 
$F_1 \otimes F_2 \subset N(\cS)$. We now apply Proposition~\ref{prop:toruschar2}
to the torus $U=\CR(\cS')$ to show that the intersection with $T_{\mathcal{S}, E}$ is one-dimensional.
\par
If the stratum $\cS$ is reducible, we have $\dim \cS = 4$ and $\dim A_\cS = 5$.
Again we must show that $T_{\mathcal{S}, E} \cap \CR(\mathcal{S}')$ is at most one-dimensional.  
In the $(2,2)$-case $\dim (F_1 \otimes F_2) = 4$, hence the codimension
of $\CR(\cS')$ in the ambient torus $A_\cS$ is at least $4$ and we obtain
immediately a contradiction.  In the $(1,3)$-case  $\dim (F_1 \otimes F_2) = 3$, 
hence $\CR(\cS')$ has to be a $2$-dimensional torus to which we apply again 
Proposition~\ref{prop:toruschar2}.
Since $\dim \ev(F_1 \otimes F_2) = 3 > 1$, we again have a contradiction. 
\end{proof}
\bibliography{my}

\end{document}